\documentclass[12pt]{article}
\usepackage{amsmath}
\usepackage{amsfonts}
\usepackage{amsthm}
\usepackage{amssymb, setspace}
\usepackage{color,graphicx,epsfig,geometry,hyperref,fancyhdr}
\usepackage[T1]{fontenc}
\usepackage{float}
\usepackage{subfig}
\usepackage{bbm}
\usepackage{mathrsfs,fleqn}
\newtheorem{theorem}{Theorem}[section]

\newcommand{\N}{\mathbb{N}}

\newcommand{\R}{\mathbb{R}}

\newcommand{\grad}{\nabla}

\newcommand{\ov}{\overline}

\begin{document}

\begin{flushleft}
\Large 
\noindent{\bf \Large Dirichlet spectral-Galerkin approximation method for the simply supported vibrating plate eigenvalues}
\end{flushleft}

\vspace{0.2in}

{\bf  \large Isaac Harris}\\
\indent {\small Department of Mathematics, Purdue University, West Lafayette, IN 47907 }\\
\indent {\small Email: \texttt{harri814@purdue.edu}}

\vspace{0.2in}

\begin{abstract}
\noindent In this paper, we analyze and implement  the Dirichlet spectral-Galerkin method for approximating simply supported vibrating plate eigenvalues with variable coefficients. This is a Galerkin approximation that uses the approximation space that is the span of finitely many Dirichlet eigenfunctions for the Laplacian. Convergence and error analysis for this method is presented for two and three dimensions. Here we will assume that the domain has either a smooth or Lipschitz boundary with no reentrant corners. An important component of the error analysis is Weyl's law for the Dirichlet eigenvalues. Numerical examples for computing the simply supported vibrating plate eigenvalues for the unit disk and square are presented. In order to test the accuracy of the approximation, we compare the spectral-Galerkin method to the separation of variables for the unit disk. Whereas for the unit square we will numerically test the convergence rate for a variable coefficient problem.
\end{abstract}



\section{Introduction}
We are interested in the numerical approximation of the simply supported vibrating plate eigenvalues. In order to compute the eigenvalues, we will employ a Galerkin approximation method. The eigenvalue problem we consider here corresponds to a fourth-order partial differential equation. Fourth-order problems have many applications in physics and engineering. A fourth-order problem was used to model the scattering by an impenetrable obstacle in an infinite elastic plate in \cite{lsm-elastic}. Here the corresponding elastic model can be reduced to a problem with the bilaplacian which is used to analyze the direct and inverse scattering problem of recovering the obstacle. The Linear Sampling Method was analyzed and implemented to recover the obstacle. In the case of acoustic scattering, the transmission eigenvalue problem is a fourth-order eigenvalue problem that has been studied by many researchers. It is well known (see for e.g. \cite{TE-book,cchlsm,armin}) that these eigenvalues can be determined from the scattering data. This leads one to study the inverse spectral problem of determining the refractive index from the eigenvalues. By using the monotonicity of the transmission eigenvalues with respect to the refractive index has been used to estimate it's average value from the eigenvalues (see for e.g. \cite{spectraltev2,te-homog}). This eigenvalue problem is associated with the scattering by an incident plane wave by an inhomogeneous isotropic media. See the manuscript \cite{TE-book} for a detailed discussion of the theory and application of the transmission eigenvalues. We also note that in \cite{zi-te} the clamped plate eigenvalues are connected to the scattering of an inhomogeneous medium.

In order to approximate fourth-order eigenvalue problems such as the transmission eigenvalues, one can use conforming finite elements such as the Argyris elements used in \cite{fem-te}. Since the Argyris elements form a globally continuously differentiable finite element space this makes them difficult to implement. Another finite element method that has been employed is the continuous interior penalty method. This uses non-conforming finite elements based on standard piece-wise continuous elements typically used for second-order partial differential equations. This method was implemented in \cite{fem-te2} for the transmission eigenvalue problem and \cite{fem-4th1} for classical biharmonic eigenvalue problems (we also refer to \cite{fem-4thandTE}). In \cite{spectraltev1,spectraltev2,spectraltev3} spectral-Galerkin methods are used to compute the transmission eigenvalues whereas in \cite{spec-stek} this method is studied for the biharmonic Steklov eigenvalue problem. Recently, in \cite{spec-4th} a new spectral element method for fourth-order problems was developed. This forms global basis functions that are usually associated with a differential operator that has sufficient approximation properties in the solution space. Since the basis functions are usually known analytically this makes these methods simple to implement. We refer to the manuscript \cite{eig-FEM-book} for the analysis of finite element methods applied to eigenvalue problems. We also note that recently the so-called method of fundamental solutions which can be seen as a boundary integral equation method has been used to compute the transmission eigenvalues in \cite{mfs-te} for a constant refractive index(see also \cite{kleefeldITP}). This method may also be able to be used to biharmonic eigenvalue problems with constant coefficients when the fourth-order differential operator can be factorized with second-order elliptic operators.

Here we are motivated by the previous numerical methods to study the so-called `Dirichlet spectral-Galerkin' method for the simply supported vibrating plate eigenvalues. This method uses the Dirichlet eigenfunctions for the Laplacian as the basis for the computational space. It is know that these eigenfunctions form an orthonormal basis in $L^2(D)$(Chapter 9 of \cite{team-pancho}) and in \cite{zite} they have been shown to have approximation properties in $H^2(D) \cap H^1_0(D)$.  This approximation method is used to compute the zero-index transmission eigenvalues with a conductive (i.e. Robin) boundary condition.  We now define the simply supported vibrating plate eigenvalues as the values $\tau \in \R_{+}$ such that there exists a nontrivial solution $u \in H^2(D)$  satisfying 
\begin{align}\label{eig1}
\Delta \alpha(x) \Delta u = \tau \beta(x) u \,\,\, \textrm{ in } \,\,\, D \quad \textrm{ with } \quad u=\Delta u=0 \,\,\, \textrm{ on } \,\,\, \partial D. 
\end{align}
For our analysis, we will assume that the domain $D \subset \R^d$ (for $d=2,3$) is a bounded simply connected open set where the boundary $\partial D$ is either polygonal with no reentrant corners or class $\mathscr{C}^2$. Throughout the paper we will also assume that $\alpha$ and $\beta$  are in $L^{\infty} (D)$ such that there exists positive constants
$$ \alpha_{\text{min}}  \leq \alpha(x) \leq \alpha_{\text{max}}  \quad \text{ and } \quad \beta_{\text{min}} \leq \beta(x)\leq \beta_{\text{max}}\quad \text{ a.e.} \, \, \,\, x \in  D.$$
 The restriction on the boundary is needed so that we may appeal to global  elliptic regularity results. We note that in recent years this problem has been studied in \cite{fem-4th1} where the continuous interior penalty method is used and in \cite{fem-4th3} where an adaptive Ciarlet-Raviart mixed finite element method was used to compute the eigenvalues. The main contribution of this paper is the study error analysis of the Dirichlet spectral-Galerkin method for this fourth-order eigenvalues problem. As seen in \cite{zite} this approximation method gives good results for a modest amount of basis functions. The idea is that by beginning with this simpler eigenvalue problem we can then apply this method to more complex eigensystems such as the transmission eigenvalue problem. 

We now summaries the preceding sections of this paper. In Section \ref{problem-statement},  we will rigorously define the variational formulation of the simply supported vibrating plate eigenvalue problem \eqref{eig1}. The analysis of the corresponding source problem is also considered in this section, which will be used in the error analysis in the following sections. Next, in Section \ref{conv-analysis} we study the convergence and error analysis for the proposed Dirichlet spectral-Galerkin method. This will require further study of the approximation properties of the Dirichlet eigenfunctions. Then, in Section \ref{numerics} we will present numerical examples of the approximation method by computing the eigenvalues in the unit disk and square. Lastly, we provide a summary of the material presented in Section \ref{last}. 

\section{The Simply Supported Plate Eigenvalues}\label{problem-statement}

In this section, we will study the variational formulation of the simply supported vibrating plate eigenvalue problem \eqref{eig1}. The analysis in this section will focus on the corresponding source problem. This amounts to the study of the solution operator for the auxiliary source problem. In order to obtain the variational formulation of \eqref{eig1}, we will appeal to Green's second Theorem. Therefore, we see that the essential boundary condition $u=0$ on $\partial D$ must be built into the solution space.  Whereas the natural boundary condition $\Delta u=0$ on $\partial D$ is used in the application of Green's second Theorem. To this end, we let $H^2(D)$ and $H^1_0(D)$ denote the standard Sobolev spaces  
equipped with their standard norms. This implies that the variational formulation of \eqref{eig1} is given by finding the values $\tau \in \R_{+}$ such that there exists a nontrivial solution $u \in H^2(D) \cap H^1_0(D)$  satisfying 
$$\int\limits_{D} {\alpha}(x) \Delta u \, \Delta \ov{v}\, \text{d}x =\tau \int\limits_{D} {\beta}(x) u \ov{v} \, \text{d}x \quad \text{ for all } \,\,\, v \in H^2(D) \cap H^1_0(D).$$
It is clear that, by appealing to the well-posedness estimate for the Poisson problem and the global $H^2$ elliptic regularity estimate we can conclude that for the Hilbert space $H^2(D) \cap H^1_0(D)$ the functional $\|\Delta \cdot \|_{L^2(D)}$ is equivalent to the standard  norm in $H^2(D)$. Therefore, we define the variational space  
$$X(D)=H^2(D) \cap H^1_0(D) \quad \text{ such that } \quad \| \cdot \|_{X(D)} = \|\Delta \cdot \|_{L^2(D)}$$
which defines a Hilbert space with the associated inner-product. We now define bounded sesquilinear forms on $X(D)$ are given by 
\begin{align}
a(u, v)= \int\limits_{D} {\alpha}(x) \Delta u \, \Delta \ov{v}\, \text{d}x \quad \text{ and } \quad b( u, v)=\int\limits_{D} \beta(x) u \ov{v} \, \text{d}x  \label{forms}
\end{align} 
for all $u$ and $v$ in $X(D)$. By the definition of the sesquilinear forms we have that the equivalent variational form of the eigenvalue problem \eqref{eig1} can be written as: find the values $\tau \in \R_{+}$ such that there exists a nontrivial solution $u \in X(D)$  satisfying 
\begin{align}
a(u,v)= \tau  b(u,v)  \quad \text{ for all } \,\,\, v \in X(D). \label{eig-varform} 
\end{align} 

Notice, that due to the assumptions on $\alpha$ we have that $a(\cdot \, , \cdot)$ is a coercive and Hermitian sesquilinear form on $X(D)$. We also have that by the assumptions on $\beta$ that $b(\cdot \, , \cdot)$ is a Hermitian sesquilinear form on $X(D)$. Now, we consider the source problem corresponding to \eqref{eig-varform} and by the Lax-Milgram Lemma we can define the solution operator 
\begin{align}
T: {X}(D) \to X(D) \,\,\, \text{ such that } \,\,\,  a\big(Tf,v \big)=b(f,v)  \quad \text{ for all } \,\,\, f, v \in X(D). \label{def-T}
\end{align} 
Notice, that since $a(\cdot \, , \cdot)$ it is an equivalent inner-product on $X(D)$. We have that since $\alpha$ and $\beta$ are real-valued 
$$a\big(Tf,v\big)=b(f,v)=\ov{b(v,f)}= \ov{a\big(Tv,f\big)}=a\big(f,Tv\big)  \quad \text{ for all } \,\,\, f, v \in X(D).$$
This implies that the solution operator $T$ is $a(\cdot \, , \cdot)$ self-adjoint.
It is clear by Green's second Theorem that $Tf $ satisfies the equation 
$$\Delta \alpha(x) \Delta Tf =  \beta(x) f \,\,\, \textrm{ in } \,\,\, D \quad \textrm{ with } \quad Tf=\Delta Tf =0 \,\,\, \textrm{ on } \,\,\, \partial D$$
and by the well-posedness satisfies  the estimate 
$$ \| Tf \|_{{X}(D)} \leq C \| f \|_{L^2(D)} \quad \text{ for all } \,\, f \in X(D)$$
which implies that $T$ is compact. Therefore, the Hilbert-Schmidt Theorem gives that there is an infinite sequence of eigenvalues $\mu_j \in \R_{+}$ and eigenfunctions $u_j \in X(D)$ for the operator $T$. This gives that there are infinitely many eigenvalues $\tau_j = \mu_j^{-1} \in  \R_{+}$ satisfying \eqref{eig-varform}. We will assume that the eigenfunctions $u_j$ are normalized with respect to the $L^2(D)$ norm.  
We now prove a regularity estimate of the bilapacian for the operator $T$ provided that $\alpha$ has bounded weak derivatives up to second-order i.e. is in the Sobolev space 
$$W^{2,\infty}(D) = \big\{ \varphi \in L^\infty(D) \, : \, \partial_{x_i} \varphi \quad  \text{and} \quad \partial_{x_i x_j} \varphi \in L^\infty(D) \, \text{ for } \, i,j =1, \cdots, d \big\} .$$

\begin{theorem}\label{regularity-T}
Let the operator $T: X(D) \to X(D)$ be as defined by \eqref{def-T}. Assume that $\alpha \in W^{2,\infty}(D)$ then $\Delta^2 Tf \in L^2(D)$ for all $f \in X(D)$ satisfying
$$\| \Delta^2 Tf \|_{L^2(D)} \leq C \| f \|_{L^2(D)}.$$
\end{theorem}
\begin{proof}
To prove the claim, notice that by the boundary value problem associated with \eqref{def-T} we have that $\Delta \alpha \Delta Tf =  \beta f $ in $D$ with zero trace on the boundary. This implies that 
$$ \| \alpha \Delta Tf \|_{H^1(D)} \leq C \| f \|_{L^2(D)}$$ 
by  the well-posedness estimate for the Poisson problem in $H^1(D)$. 
Using the identity 
$$\grad\left(\alpha \Delta Tf \right) = (\Delta Tf ) \grad \alpha  + \alpha \grad ( \Delta Tf)$$
we have that $\| \Delta Tf \|_{H^1(D)} \leq C  \| f \|_{L^2(D)}$ since $\alpha \in W^{2,\infty}(D)$ and bounded below. Now we use that 
$$ \Delta \alpha \Delta Tf = \alpha \Delta^2 Tf +2 \grad \alpha \cdot \grad(\Delta Tf) +(\Delta Tf) \Delta\alpha$$
we similarly obtain the $\| \Delta^2 Tf \|_{L^2(D)} \leq C \| f \|_{L^2(D)}$ by again using the fact that $\alpha \in W^{2,\infty}(D)$ and bounded below. 
\end{proof}

\section{Analysis of the Approximation}\label{conv-analysis} 
In this section, we will analyze the proposed numerical method for approximation of the solutions to \eqref{eig-varform}. This method uses a conforming computational space to compute the Galerkin approximation of the eigenvalues. The computational space is taken to be the span of finitely many Dirichlet eigenfunctions for the Laplacian. In \cite{zite} this method is used to approximate eigenvalues associated with a fourth-order problem with a Robin condition, which shows the versatility of this approximation method.

 \subsection{Analysis of the Approximation Space}\label{approx}
 This part of the paper is dedicated to the approximation space. We will first define the approximation space as well as give some known result for it's approximation property in $X(D)$. To this end, our spectral approximation space we will take the conforming subspace of $X(D)$ be given by 
 $$X_N (D) = \text{span}\big\{ \phi_j (x) \big\}_{j=1}^N \quad \text{ for some fixed } \quad N \in \N.$$
Here the basis functions $\phi_j$ are the $j$th Dirichlet eigenfunction for the Laplacian and the corresponding eigenvalue $\lambda_j \in \R_{+}$ for the domain $D$. We arrange the sequence of Dirichlet eigenvalues in non-decreasing order such that $0 < \lambda_j \leq \lambda_{j+1}$ for all $j \in\N$. Recall, that the Dirichlet eigenpairs $(\lambda_j  , \phi_j)_{j=1}^{\infty}  \in \R_{+} \times H^1_0(D)$ satisfy
\begin{align}
- \Delta \phi_j = \lambda_j \phi_j \,\, \text{ in } \,\, D \quad \text{ where }  \quad  \| \phi_j\|_{L^2(D)}=1. \label{dirichlet-eig}
\end{align}
By elliptic regularity, we have that the basis functions $\phi_j \in X(D)$ and by Theorem 3.7 in \cite{two-eig-cbc} we have that they form an orthogonal basis for the variational space $X(D)$. Recall, the Weyl's law estimate (see for e.g. \cite{weyl-law}) that there exist two constants $c_1 , c_2 >0$ independent of $j$ such that 
 $$c_1 j^{2/d} \leq \lambda_j \leq c_2 j^{2/d}\quad \text{ for all } \quad j \gg 1.$$
 This will be used to determine the error estimates for our computational space. 
Since the Dirichlet eigenfunctions $\{ \phi_j \}_{j=1}^\infty$ form an orthonormal basis of $L^2(D)$ we have that 
\begin{align}
f = \sum\limits_{j=1}^{\infty} (f,\phi_j)_{L^2(D)} \phi_j \quad \text{ giving that} \quad \Delta^m f  = \sum\limits_{j=1}^{\infty} (-\lambda_j)^m (f,\phi_j)_{L^2(D)} \phi_j.\label{fouier-series}
\end{align}
for all $f \in L^2(D)$. We can now define the domain of the $m$th power of the Laplacian as the subset  of $L^2(D)$ given by  
$$\mathscr{D} \big( \Delta^m \big) :=\big\{ f \in L^2(D) \, : \,  \Delta^m f \in L^2(D)  \,\, \text{as defined by equation \eqref{fouier-series}}\big\}. $$ 
Therefore, by the definition of $\mathscr{D} \big( \Delta^m \big)$ we can conclude that it is a Hilbert space with the associated norm  
\begin{align}
\| f \|^2_{\mathscr{D} ( \Delta^m )} = \sum\limits_{j=1}^{\infty} \lambda_j^{2m} \big|(f,\phi_j)_{L^2(D)}\big|^2 <\infty. \label{laplacian-m}
\end{align} 
Theorem 3.1 in \cite{zite} gives that $X(D) \subseteq \mathscr{D} ( \Delta )$ and $\|  \cdot \|_{X(D)}=\| \cdot \|_{\mathscr{D} ( \Delta )}$. By this fact we can conclude that the $L^2(D)$ projection onto the approximation space $X_N (D)$ denote by $\Pi_N :X(D) \to X_N (D)$  is given by
 $$\Pi_N f =  \sum\limits_{j=1}^{N}  (f,\phi_j)_{L^2(D)} \phi_j \quad \text{for some fixed} \quad N \in \N$$ 
converges point-wise on $X(D)$. 
We now prove that for sufficiently smooth $\alpha$ the range of $T$ is contained in $\mathscr{D} ( \Delta^2 )$.
\begin{theorem}\label{T-subset}
Let the operator $T$ be as defined in \eqref{def-T}. Assume that $\alpha \in W^{2,\infty}(D)$ then Range$(T)\subseteq \mathscr{D} ( \Delta^2 )$ such that $\| Tf \|_{\mathscr{D} ( \Delta^2 )} \leq C \| f \|_{L^2(D)}$. 
\end{theorem}
\begin{proof}
In order to prove the claim we first note that for $\alpha \in W^{2,\infty}(D)$ Theorem \ref{regularity-T} implies that any $u \in \text{Range}(T)$ satisfies 
$$u \in X(D) \,\,\, \text{ such that} \,\, \, \Delta^2 u  \in L^2(D) \,\,\, \text{and} \,\,\, \Delta u =0 \,\,\, \text{on}\,\,\,  \partial D.$$ 
Therefore, notice that by applying Green's second Theorem twice we have that 
$$ (\Delta^2 u,\phi_j)_{L^2(D)} = ( u, \Delta^2 \phi_j)_{L^2(D)}= \lambda_j^2( u,  \phi_j)_{L^2(D)} \quad \text{ for all } \quad j \in \N.$$
We then have that 
$$\| \Delta^2 u \|_{L^2(D)}^2 = \sum\limits_{j=1}^{\infty} \big|(\Delta^2 u,\phi_j)_{L^2(D)}\big|^2 = \sum\limits_{j=1}^{\infty} \lambda_j^{4} \big|(u,\phi_j)_{L^2(D)}\big|^2$$
which implies that $\| \Delta^2 u \|_{L^2(D)}=\| u \|_{\mathscr{D} ( \Delta^2 )}$. Proving the claim.
\end{proof}

\subsection{Convergence and Error Analysis}\label{conv-eig}
This section is dedicated to studying the convergence analysis and error estimates for the Dirichlet spectral-Galerkin method for approximating the simply supported vibrating plate eigenvalues. Here we will use the variational formulation along with the analysis of the approximation space $X_N (D)$.

To begin, we will first need to prove that the Galerkin approximation of the solution operator converges in norm as $N \to \infty$. Now define the Galerkin approximation of the solution operator as $T_N: X(D) \to X_N(D)$ such that for any $f \in X(D)$  
\begin{align}
a\big(T_Nf,v_N \big)=b(f,v_N)  \quad \text{ for all } \,\,\, v_N \in X_N(D). \label{def-TN}
\end{align} 
Since, Range$(T_N) \subseteq X_N (D)$ we have that $T_N$ has finite rank and is therefore compact. Simple calculations give that $T_N$ restricted to the computational space $X_N(D)$ is an $a( \cdot \, , \, \cdot)$ self-adjoint operator on the Hilbert space $X_N(D)$. Therefore, the approximation of the solution operator $T_N$ has $N$ eigenvalues denoted $\mu_{j,N}>0$ and corresponding eigenfunctions $u_{j,N} \in X_N(D)$ for $j\in \{ 1, \cdots , N\}$ counting multiplicity. This implies that $u_{j,N} \in X_N(D)$ satisfying
\begin{align}
a(u_{j,N},v_N)=\tau_{j,N} b(u_{j,N},v_N)  \quad \text{ for all } \,\,\, v_N \in X_N(D) \label{spec-varform} 
\end{align} 
where $\tau_{j,N} = \mu_{j,N}^{-1}$ and $\| u_{j,N} \|_{L^2(D)}=1$.  It is clear that $(\tau_{j,N} , u_{j,N} )_{j=1}^N \in \R_{+} \times X_N(D)$ are the $j$th Galerkin approximation of the vibrating plate eigenpair satisfying \eqref{eig-varform}. 

By appealing to the analysis in \cite{osborn} we will have the convergence provided that $T_N$ converges to $T$ in the operator norm  as $N \to \infty$. 
Therefore, by the standard Galerkin orthogonality, we may appeal to Cea's Lemma (Chapter 9 of \cite{numerics-book}) to obtain  
$$\big\| Tf - T_N f \big\|_{X(D)} \leq C \big\| Tf - v_N \big\|_{X(D)} \quad \text{ for any } \,\,\, v_N  \in X_N(D)$$
and for all $f \in X(D)$. Since $\Pi_NT f \in X_N(D)$ for all $f \in X(D)$ where $\Pi_N$ is the $L^2(D)$ projection onto $X_N(D)$ we have that 
$$\big\| Tf - T_N f \big\|_{X(D)} \leq C  \big\| (I-\Pi_N)T f \big\|_{X(D)}$$
This estimate will be used to prove convergence and error estimates for the approximation of the eigenvalues. 

\begin{theorem} \label{discrete-eig-conv}
Let $(\tau_j , u_j)$ and $(\tau_{j,N} , u_{j,N})$ be the $j$th eigenpairs for \eqref{eig-varform} and \eqref{spec-varform} respectively. Then as $N \to \infty$ we have that $\tau_{j,N} \to \tau_{j}$ and $u_{j,N} \to u_{j}$ in $X(D)$.  
\end{theorem}
\begin{proof}
Recall that the operator $T$ is compact. Therefore, we have that the point-wise convergence of $\Pi_N$ to the identity operator on $X(D)$ implies the norm convergence i.e. $\big\| (I-\Pi_N)T\big\|_{X(D) \mapsto X(D)}\to 0$ as $N \to \infty$ (see for e.g \cite{FEM-an}). The result follows from 
$$\big\|T -T_N \big\|_{X(D) \mapsto X(D)} \leq  C \big\| (I-\Pi_N)T  \big\|_{X(D)  \mapsto X(D)} \to 0 \quad \text{ as } \quad N \to \infty$$
and then appealing to the results in \cite{osborn}. 
\end{proof}
 
To determine the convergence rate we use the fact that for $(\tau_j , u_j)$ and $(\tau_{j,N} , u_{j,N})$ being the $j$th eigenpairs for \eqref{eig-varform} and \eqref{spec-varform} then   
\begin{align}
a\big(u_{j,N}-u_{j} , u_{j,N} -u_{j}\big) -\tau_j  b\big(u_{j,N}-u_{j} , u_{j,N} -u_{j}\big) = \big( \tau_{j,N} -\tau_{j} \big) b\big(u_{j,N} , u_{j,N}\big)\label{eig-difference}
\end{align}
for any $N \in \N$. By the boundedness of the sesquilinear forms $a( \cdot \, , \, \cdot)$ and $b( \cdot \, , \, \cdot)$ defined by \eqref{forms} along with the estimate $\beta \geq \beta_{\text{min}}$ we have that 
$$\big| \tau_{j,N} - \tau_j \big| \leq C(1+ \tau_j) \big\| u_{j,N} -u_{j} \big\|^2_{X(D)}$$
where we have used the fact that $\| u_{j,N} \|_{L^2(D)}=1$. 
Notice, that we can appeal to the results in \cite{babuska-osborn}, which gives that 
$$\| u_{j,N} -u_{j} \|_{X(D)} \leq C \big\|(T - T_N) \big\|_{E(\tau_j) \mapsto X(D)}.$$
Here $E(\tau_j)$ is the finite-dimensional eigenspace corresponding to $\tau_j$ where $u \in E(\tau_j)$ satisfies $Tu = \tau_j^{-1} u$.  Therefore, to estimate the convergence rate of the eigenvalues we need to estimate the convergence rate in the operator norm of $T_N$ to $T$ as $N \to \infty$. 
To do so, recall that by  Cea's Lemma, we have that 
\begin{align*}
\big\|(T -T_N) \big\|^2_{E(\tau_j) \mapsto X(D)} &=\sup\limits_{u \in E(\tau_j) \, : \,  \| u \|_{X(D)} =1} \big\|(T -T_N)u \big\|^2_{X(D)}\\
			&\leq C \sup\limits_{u \in E(\tau_j) \, : \,  \| u \|_{X(D)} =1} \big\|(I -\Pi_N)Tu \big\|^2_{X(D)}.
\end{align*}
We recall that $\Pi_N :X(D) \to X_N (D)$ denotes the $L^2(D)$ projection operator via the Fourier-Series with the Dirichlet eigenfunctions. By definition we have that 
\begin{align*}
\big\|(T -T_N) \big\|^2_{E(\tau_j) \mapsto X(D)} &\leq C \sup\limits_{u \in E(\tau_j) \, : \,  \| u \|_{X(D)} =1} \sum\limits_{j=N+1}^{\infty} \lambda_j^{2} \big|(Tu,\phi_j)_{L^2(D)}\big|^2\\
			&\leq  \frac{C}{ \lambda_{N+1}^{2(m-1)} } \, \sup\limits_{u \in E(\tau_j) \, : \,  \| u \|_{X(D)} =1}  \sum\limits_{j=N+1}^{\infty} \lambda_j^{2m} \big|(Tu,\phi_j)_{L^2(D)}\big|^2.  
\end{align*}
We now assume that the eigenspace $E(\tau_j) \subseteq \mathscr{D} ( \Delta^m)$ then we have that 
\begin{align*}
\big\|(T -T_N) \big\|^2_{E(\tau_j) \mapsto X(D)} &\leq \frac{C\tau_j^{-2}}{\lambda_{N+1}^{2(m-1)} } \, \,  \sup\limits_{u \in E(\tau_j) \, : \,  \| u \|_{X(D)} =1} \sum\limits_{j=N+1}^{\infty}  \lambda_j^{2m} \big|(u,\phi_j)_{L^2(D)}\big|^2 \\
								 & \leq \frac{C}{ (N+1)^{4(m-1)/d}} \, \sup\limits_{u \in E(\tau_j) \, : \,  \| u \|_{X(D)} =1} \big\| (I-\Pi_N) u \big\|^2_{\mathscr{D} ( \Delta^m)}
\end{align*}
where we have used the Weyl's law estimate. Given the above calculations we have the following convergence estimate for the Dirichlet spectral-Galerkin approximation.

\begin{theorem} \label{eig-convrate}
Let $\tau_{j}$ and $\tau_{j,N}$ be the $j$th eigenvalues for \eqref{eig-varform} and \eqref{spec-varform} respectively. Assume that the eigenspace $E(\tau_{j}) \subseteq \mathscr{D} ( \Delta^m)$ for some $m \in \N$ then
$$\big| \tau_{j,N}-\tau_{j} \big| \leq \frac{C}{ (N+1)^{4(m-1)/d}} \sup\limits_{u \in E(\tau_j) \, : \,  \| u \|_{X(D)} =1} \big\| (I-\Pi_N) u \big\|^2_{\mathscr{D} ( \Delta^m)}$$
where the constant $C>0$ is independent of $N$ and $\Pi_N:X(D) \to X_N (D)$ is the $L^2(D)$ projection onto $X_N (D)$.
\end{theorem}
\begin{proof}
To prove the claim, we recall that from the above analysis we have that 
$$\big| \tau_{j,N} - \tau_j \big| \leq C(1+ \tau_j) \big\| u_{j,N} -u_{j} \big\|^2_{X(D)}$$
as well as 
$$\| u_{j,N} -u_{j} \|^2_{X(D)} \leq \frac{C}{ (N+1)^{4(m-1)/d}} \sup\limits_{u \in E( \tau_j) \, : \,  \| u \|_{X(D)} =1} \big\| (I-\Pi_N) u \big\|^2_{\mathscr{D} ( \Delta^m)}.$$
This proves the result.
\end{proof} 
 
It is clear that the above analysis also gives a convergence rate for the corresponding eigenfunctions. For completeness, we now state the error estimate for the eigenfunctions in the following result. Due to the continuous embedding of $H^2(D)$ in $\mathscr{C}({D})$ the following result gives an $L^{\infty}(D)$ norm estimate of the error for the eigenfunctions. 
 
 \begin{theorem} \label{eigfunc-convrate}
Let $u_{j}$ and $u_{j,N}$ be the $j$th eigenfunctions for \eqref{eig-varform} and \eqref{spec-varform} respectively. Assume that the eigenspace $E(\tau_{j}) \subseteq \mathscr{D} ( \Delta^m)$ for some $m \in \N$ then 
$$ \big\| u_{j,N}-u_{j} \big\|_{X(D)}  \leq \frac{C}{ (N+1)^{2(m-1)/d}} \sup\limits_{u \in E(\tau_j) \, : \,  \| u \|_{X(D)} =1} \big\| (I-\Pi_N) u \big\|_{\mathscr{D} (\Delta^m)}$$
where the constant $C>0$ is independent of $N$ and $\Pi_N :X(D) \to X_N (D)$ is the $L^2(D)$ projection onto $X_N (D)$. 
\end{theorem}
\begin{proof}
The result follows from the above analysis.
\end{proof}
 
From Theorems \ref{eig-convrate} and \ref{eigfunc-convrate} we see that the convergence result depends on if the eigenspace is contained in $\mathscr{D} ( \Delta^m)$. It is clear that provided $\alpha  \in W^{2,\infty}(D)$ then $E(\tau_{j}) \subseteq \mathscr{D} (\Delta^2)$ by Theorem \ref{T-subset} for each $j \in \N$. Therefore, Theorem \ref{eig-convrate} gives that the convergence rate for the eigenvalues is at least $4/d$ where $d=2,3$. The convergence rate will be studied numerically in the coming section.  

\begin{theorem} \label{constant-conv}
Let $(\tau_j , u_j)$ and $(\tau_{j,N} , u_{j,N})$ be the $j$th eigenpairs for \eqref{eig-varform} and \eqref{spec-varform} respectively. Assume that $\alpha \in W^{2,\infty}(D)$  then
$$\big| \tau_{j,N}-\tau_{j} \big| \leq \frac{C}{ (N+1)^{4/d}} \quad \text{ and } \quad  \big\| u_{j,N}-u_{j} \big\|_{X(D)}  \leq \frac{C}{ (N+1)^{2/d}}$$
where the constant $C>0$ is independent of $N$.
\end{theorem}
\begin{proof}
The result follows directly from the estimates in Theorems \ref{eig-convrate} and \ref{eigfunc-convrate} and by appealing to Theorem  \ref{T-subset}.
\end{proof}
 
\section{Numerical Examples}\label{numerics}   
 In this section, we provide some numerical examples for computing the simply supported vibrating plate eigenvalues for the uint circle and unit sphere. The purpose is to validate the Dirichlet spectral-Galerkin approximation method for computing the eigenvalues. For the domains under consideration, the Dirichlet eigenpairs are known via separation of variables. In order to apply this method, one needs to know the Dirichlet eigenpairs a priori. Therefore, one needs to pre-calculate the Dirichlet eigenpairs for the domain $D$, which can be done by the BEM \cite{BEM-dirichlet} or FEM \cite{eig-FEM-book}. The numerical experiments are all done with the software \texttt{MATLAB} 2020a on a MacBook Pro with a 2.3 GHz Intel Core i5 processor with 8GB of memory.

 In the following numerical examples, we will apply the Dirichlet spectral-Galerkin method to approximate the eigenvalue $\tau$ of  \eqref{eig-varform}. To do so, we will let the approximate vibrating plate eigenfunctions $u_N$ as the span of the finitely many Dirichlet eigenfunctions denote $\phi_i$ satisfy \eqref{dirichlet-eig}. By using this approximation in \eqref{spec-varform} with $v_N = \phi_j$ we obtain the matrix eigenvalue problem given by 
 \begin{align}
\left( {\bf A}-\tau_N{\bf B}\right) \vec{u}_N =0 \quad \text{ where } \quad \vec{u}_N \neq \vec{0}. \label{g-eig}
\end{align}
Here the Galerkin matrices are given by 
$${\bf A}_{i,j}  = a(\phi_i , \phi_j) \quad \text{ and } \quad {\bf B}_{i,j} =  b(\phi_i , \phi_j) \quad \text{ for } \,\, i,j =1, \cdots , N.$$
It is clear that the corresponding eigenvalues $\tau_{j,N}$ satisfy the Galerkin eigenvalue problem \eqref{spec-varform}. Notice that by the definition of the sesquilinear forms we have that 
$$a(\phi_i , \phi_j) =  \int\limits_{D} \alpha(x) \,   \Delta \phi_i (x) \,  \Delta \overline{\phi}_j(x) \, \text{d}x  \quad \text{and} \quad b(\phi_i , \phi_j) =  \int\limits_{D} \beta(x) \,  \phi_i (x) \, \overline{\phi}_j(x) \, \text{d}x.$$
Then the matrix eigenvalue problem \eqref{g-eig} is solved by appealing to the built-in eigenvalue solver i.e. the `\texttt{eig}' command in \texttt{MATLAB}.

\subsection{Examples for the Unit Disk}\label{eig-approx1}
Here, we will provide some numerical examples for the unit disk. To this end, we will write the integral to compute the Galerkin matrices in polar coordinates. To compute the Galerkin matrices we implement a 2d Gaussian quadrature method. We use a 12 point quadrature scheme to evaluate the integrals in each radial and angular directions. Here the approximation space is given by the span of finitely many basis functions 
$$\phi_{j}(r,\vartheta) = J_{n} \left(\sqrt{\lambda_{n,m}} \, r\right) \cos(n \vartheta)\quad \text{with index} \quad j=j(n,m) \in \N.$$
The square root of the Dirichlet eigenvalues $\sqrt{\lambda_{n,m}}$ corresponds to the $m$th positive root of the $n$th first kind Bessel function denoted $J_n$ for all $n\in \N \cup\{0\}$ and $m \in \N$. In order to compute the roots $\sqrt{\lambda_{n,m}}$ we use the `\texttt{fzero}' command in \texttt{MATLAB}. 
In the following numerical examples we take 24 basis functions where $0\leq n \leq 5$ and $1\leq m \leq 4$ which will give that the approximation space is defined as 
$$ {X}_N (D) = \text{Span} \Big\{ J_{n} \left(\sqrt{\lambda_{n,m}} \, r\right) \cos(n \vartheta)  \Big\}_{n=0\, ,\, m=1}^{n=5 \, ,\, m=4}.$$
We check the relative error compared to separation of variables for the constant coefficient case. It can be shown that the exact simply supported vibrating plate eigenvalues $\tau_j$ are given by the solutions to the transcendental equations 
\begin{align}
J_{n} \left(\sqrt[4]{\tau {\beta}/{\alpha}}\right) =0\quad \text{for any} \quad n \in \N \cup \{0\}. \label{exact-eig}
\end{align}
By Theorem \ref{constant-conv} the computational space with 24 basis functions should be accurate on the order of $10^{-2}$ so we present some numerical examples with variable coefficients. We see the monotonicity with respect to the coefficients. Where the Courant-Fischer (see for e.g. Chapter 4 of \cite{TE-book}) min-max principle implies that computed eigenvalues $\tau_{j,N}$ are decreasing with respect to $\beta$ and increasing with respect to $\alpha$.

\subsubsection{Example 1}
\noindent{\bf Constant coefficients:} For our first example, we compute the first three eigenvalues corresponding to constant coefficients $\alpha =1/4$ and $\beta=10$ which are reported in Table \ref{eig-compare}. In order to show that the approximation works well for a limited number of basis functions we compare with the roots to \eqref{exact-eig} which are computed using the `\texttt{fzero}' command in \texttt{MATLAB}. Here we also report the relative error given by $|\tau_{j,N} - \tau_{j}|/\tau_{j}$ where $\tau_{j}$ and $\tau_{j,N}$ denote the $j$th eigenvalues for \eqref{eig-varform} and \eqref{spec-varform} respectively. \\

\begin{table}[ht!]
\centering  
\begin{tabular}{ c | c | c } 
\hline                  
     Approx        &  Exact    & Rel Error    \\ [0.5ex] 
\hline                  
\hline                  
$\tau_{1,N}=0.8361309908$ & $\tau_{1}=0.8361309971$ &  $7.532\times 10^{-9}$ \\
$\tau_{2,N}=5.3890063494$ & $\tau_{2}=5.3890065484$ &  $3.691\times 10^{-8}$ \\
$\tau_{3,N}=17.390486256$ & $\tau_{3}=17.390509792$ &  $1.353\times 10^{-6}$ \\
\hline 
\end{tabular}
\caption{Comparison of the approximated eigenvalues vs exact eigenvalues given by separation of variables for the unit disk with $\alpha =1/4$ and $\beta=10$. }\label{eig-compare}
\end{table}

\subsubsection{Example 2}
\noindent{\bf Continuous coefficients:} Here we present two numerical examples of the computation of the eigenvalues for smooth variable coefficients. In Table \ref{eig-circle2} we report the  first three computed eigenvalues. We take the two pairs of coefficients given by 
$$ \alpha = 1/4+ r \sin^2(\vartheta) \, \, ; \, \, \beta=10 \quad \text{and} \quad \alpha = 1/4 \, \, ; \, \,\beta=10+2\text{exp}(-r^2).$$
Even though the accuracy is not studied we see that the computed eigenvalues $\tau_{j,N}$ satisfy the monotonicity with respect to the parameters $\alpha$ and $\beta$. This can be seen by comparing to the values given in Table \ref{eig-compare}.
\\

\begin{table}[ht!]
\centering  
\begin{tabular}{ c | c  } 
\hline                  
     $\alpha =1/4+ r \sin^2(\vartheta)$ and $\beta=10$    &  $\alpha = 1/4$ and $\beta=10+2\text{exp}(-r^2)$   \\ [0.5ex] 
\hline                  
\hline                  
$\tau_{1,N}=1.3319675316$ & $\tau_{1,N}=0.7187824125$ \\
$\tau_{2,N}=7.4582616121$ & $\tau_{2,N}=4.7034217796$ \\
$\tau_{3,N}=29.166695881$ & $\tau_{3,N}=15.321238476$ \\
\hline 
\end{tabular}
\caption{Approximation of the eigenvalues for two sets of continuous coefficients for the unit disk. }\label{eig-circle2}
\end{table}

\subsubsection{Example 3}
\noindent{\bf Piece-wise constant coefficients:} Now we report the numerically computed eigenvalues for piece-wise constant coefficients in Table \ref{eig-circle3}. The parameters $\alpha$ and $\beta$ are taken to be radially piece-wise constant functions in the unit disk. We take the two pairs of coefficients given by 
$$ \alpha =  1/4\left(1+\mathbbm{1}_{(r>0.25)}\right) \, \, ; \, \, \beta=10 \quad \text{and} \quad \alpha = 1/4 \, \, ; \, \,\beta=10+2\mathbbm{1}_{(r>0.5)}$$
where $\mathbbm{1}_{I}$ denotes the indicator function on the interval $I$. Again, we note the computed eigenvalues $\tau_{j,N}$ satisfy the monotonicity with respect to the parameters $\alpha$ and $\beta$ by comparison with the values in Table \ref{eig-compare}.

\begin{table}[ht!]
\centering  
\begin{tabular}{ c | c  } 
\hline                  
     $\alpha = 1/4\left(1+\mathbbm{1}_{(r>0.25)}\right)$ and $\beta=10$    &  $\alpha = 1/4$ and $\beta=10+2\mathbbm{1}_{(r>0.5)}$   \\ [0.5ex] 
\hline                  
\hline                  
$\tau_{1,N}=1.2757025572$ & $\tau_{1,N}=0.7814121836$ \\
$\tau_{2,N}=9.9620153799$ & $\tau_{2,N}=4.7870466455$ \\
$\tau_{3,N}=34.214855049$ & $\tau_{3,N}=14.975919099$ \\
\hline 
\end{tabular}
\caption{Approximation of the eigenvalues for the two sets of discontinuous coefficients for the unit disk. }\label{eig-circle3}
\end{table}

\subsection{Examples for the Unit Square}\label{eig-approx2}
Here, we will provide some numerical examples for the unit square given by $D=(0,1)^2$. Our focus is on numerically determining the convergence rate. To this end, we will provide the approximation of the first three simply supported vibrating plate eigenvalues $\tau_{j,N}$ for multiple values of $N$. Here the value $N$ denotes the degrees of freedom(DoF) for the Galerkin matrices. In order to estimate the convergence rate we check the relative error given by 
$$R(i) = \frac{| \tau_{j,N_{i}} - \tau_{j,N_{i+1}}|}{\tau_{j,N_{i+1}}} \quad \text{for the first three eigenvalues with} \quad i = 1,2,3.$$
Therefore, by the relative error, we can conclude the convergence of the approximation. The computational space will be given by the span of finitely many basis functions which are given by 
$$\phi_{j}(x_1,x_2) = \sin(n\pi x_1)\sin(m\pi x_2) \quad \text{with index} \quad j=j(n,m) \in \N.$$
We take $1\leq n,m\leq M$ which gives that the computational space is 
$$ X_N (D)=\text{Span} \Big\{  \sin(n\pi x_1)\sin(m\pi x_2)  \Big\}_{n, m=1}^{n, m \leq M}.$$
  To compute the Galerkin matrices we will use the built-in `\texttt{integral2}' command in \texttt{MATLAB}. Here we provide examples for multiple pairs of parameters $\alpha$ and $\beta$.

\subsubsection{Example 4}
\noindent{\bf Constant coefficient $\alpha$:} For this example, we compute the eigenvalues $\tau_{j,N}$ for various values of $N$ with a constant $\alpha$. This is done for the parameters 
$$\alpha=1/4 \quad \text{and} \quad \beta=(x_1^2+1)(x_2^2+10)$$
where the eigenvalues are reported in Table \ref{eig-square1}. In our calculations, we see numerically that $\tau_{2,N} \approx \tau_{3,N}$. Therefore, we will  report the eigenvalues $\tau_{j,N}$ where $j=1,2,4$. In Table \ref{error-square1} we report the relative error for the reported eigenvalues. 

\begin{table}[ht!]
\centering  
\begin{tabular}{ c | c | c | c } 
\hline                  
      DoF      & $\tau_{1,N}$    &   $\tau_{2,N} $  &   $\tau_{4,N} $   \\ [0.5ex] 
\hline                  
\hline                  
$N=10$  &$7.3573520969$ & $45.421593372$  & $118.08243134$ \\
$N=15$  &$7.3570078657$ & $45.395939541$  & $116.61165938$  \\
$N=20$  &$7.3569921869$ & $45.394566012$  & $116.56679385$ \\
$N=25$  &$7.3569922082$ & $45.394456592$  & $116.56395603$\\
\hline 
\end{tabular}
\caption{The first three `distinct' approximated eigenvalues for the unit square with various values of $N$ where $\alpha=1/4$ and $\beta=(x_1^2+1)(x_2^2+10)$.  }\label{eig-square1}
\end{table}

\begin{table}[ht!]
\centering  
\begin{tabular}{  c | c | c | c  } 
\hline                  
            & Rel. error for $\tau_{1,N}$    &   Rel. error for $\tau_{2,N}$  &   Rel. error for $\tau_{4,N}$   \\ [0.5ex] 
\hline                  
\hline                  
$N=10$  &$0.000046789$ & $0.000565112$   & $0.012612563$  \\
$N=15$  &$0.000002131$ & $0.000030257$   & $0.000384891$  \\
$N=20$  &$0.000000003$ & $0.000002410$   & $0.000024345$ \\
\hline 
\end{tabular}
\caption{The relative error for the first three `distinct'  approximated eigenvalue for the unit square with various values of $N$ where  $\alpha=1/4$ and $\beta=(x_1^2+1)(x_2^2+10)$.  }\label{error-square1}
\end{table}

\subsubsection{Example 5}
\noindent{\bf Constant coefficient $\beta$:} This example is focused on computing the first three eigenvalues for case when the parameter $\beta$ is constant. In Table \ref{eig-square2} we report  the corresponding eigenvalues $\tau_{j,N}$ where $j=1,2,3$  for various values of $N$ for the parameters given by 
$$\alpha=x_1 x_2+1/4 \quad \text{and} \quad \beta=10.$$
For this set of parameters, we  also report the relative errors in Table \ref{error-square2}. 

\begin{table}[ht!]
\centering  
\begin{tabular}{  c | c | c | c } 
\hline                  
        DoF    & $\tau_{1,N}$    &   $\tau_{2,N}$  &   $\tau_{3,N}$   \\ [0.5ex] 
\hline                  
\hline                  
$N=10$  &$18.05863342659$ & $109.6130034603$    &  $116.67276901349$  \\ 
$N=15$  &$18.02755051044$ & $107.4962046649$    &  $113.77896218468$ \\ 
$N=20$  &$18.01765352838$ & $107.2367212907$    &  $113.67996470052$ \\ 
$N=25$  &$18.01607463791$ & $107.1599636478$    &  $113.64645491602$ \\ 
\hline 
\end{tabular}
\caption{The first three approximated eigenvalues for the unit square with various values of $N$ where $\alpha= x_1 x_2+1/4$ and $\beta=10$.}\label{eig-square2}
\end{table}

\begin{table}[ht!]
\centering  
\begin{tabular}{  c | c | c | c } 
\hline                  
    & Rel. error for $\tau_{1,N}$    &   Rel. error for $\tau_{2,N}$  &   Rel. error for $\tau_{3,N}$   \\ [0.5ex]  
\hline                  
\hline                  
$N=10$  &$0.001128353$ & $0.019691846$    &  $0.025433584$ \\
$N=15$  &$0.000421611$ & $0.002419724$    &  $0.000870843$ \\
$N=20$  &$0.000062390$ & $0.000716290$    &  $0.000294859$ \\
\hline 
\end{tabular}
\caption{The relative error for the first three approximated eigenvalues for the unit square with various values of $N$ where $\alpha= x_1 x_2+1/4$ and $\beta=10$.  }\label{error-square2}
\end{table}

\subsubsection{Example 6}
\noindent{\bf Non-constant coefficients:} For our last example, we compute the eigenvalues $\tau_{j,N}$ where both the parameters $\alpha$ and $\beta$ are non-constant functions. Therefore, we will see the convergence of the approximation method where we have two variable coefficients. The eigenvalues for $j=1,2,3$ are reported in Table \ref{eig-square3} for 
$$\alpha=x_1 x_2+1/4 \quad \text{and} \quad \beta=(x_1^2+1)(x_2^2+10).$$
Just as in example 5 we estimate we report the relative errors to see the convergence in Table \ref{error-square3}.  Here we also estimate the convergence rate $p$ from the relative error where $R(i) = \mathcal{O}(N_i^{-p})$ which is also reported.

\begin{table}[ht!]
\centering  
\begin{tabular}{  c | c | c | c } 
\hline                  
      DoF     & $\tau_{1,N}$    &   $\tau_{2,N}$  &   $\tau_{3,N}$   \\ [0.5ex] 
\hline                  
\hline                  
$N=10$  &$13.9832243446$ & $82.8159686896$    &$89.7089031694$ \\
$N=15$  &$13.9674641031$ & $81.7610212109$    &$89.2854919971$ \\
$N=20$  &$13.9615777461$ & $81.6259370343$    &$89.2499753721$ \\
$N=25$  &$13.9607067280$ & $81.5805519814$    &$89.2373897357$ \\
\hline 
\end{tabular}
\caption{The first three approximated eigenvalues for the unit square with various values of $N$ where $\alpha=x_1 x_2+1/4$ and $\beta=(x_1^2+1)(x_2^2+10)$.  }\label{eig-square3}
\end{table}

\begin{table}[ht!]
\centering  
\begin{tabular}{  c |  l  |  l  |  l } 
\hline                  
                & Rel. error for $\tau_{1,N}$    &   Rel. error for $\tau_{2,N}$  &   Rel. error for $\tau_{3,N}$   \\ [0.5ex]  
\hline                  
\hline                  
$N=10$  &$0.001128353$ & $0.012902816$    &  $0.004742216$ \\
$N=15$  &$0.000421611\,(p=2.43)$ & $0.001654917\,(p=5.07)$    &  $0.000397945\,(p=6.11)$ \\
$N=20$  &$0.000062390\,(p=6.64)$ & $0.000556321\,(p=3.79)$    &  $0.000141035\,(p=3.61)$ \\
\hline 
\end{tabular}
\caption{The relative error and convergence rate $p$ for the first three approximated eigenvalues for the unit square with various values of $N$ where $\alpha=x_1 x_2+1/4$ and $\beta=(x_1^2+1)(x_2^2+10)$.  }\label{error-square3}
\end{table}

For completeness, we also plot the first three eigenfunctions corresponding to example 6 in Figure \ref{eigfunc}. Recall, that the results from the previous section gives the convergence of the eigenfunctions as in Theorems \ref{eigfunc-convrate} and  \ref{constant-conv}. We also note that in our examples we see that the relative error seems to tend to zero at a rate of $\mathcal{O}(N^{-4})$ as $N \to \infty$. The analysis given in the previous section gives a convergence rate of at least second-order for the examples presented here. Also, presented in Figure \ref{eigfunc} is the convergence plot of the relative error $R(i)$ for the eigenvalues in Table \ref{error-square3}. 
\vspace{-0.2in}
\begin{figure}[H]
\centering
\subfloat[eigenfunction for $\tau_{1,N}$]{\includegraphics[width=7cm]{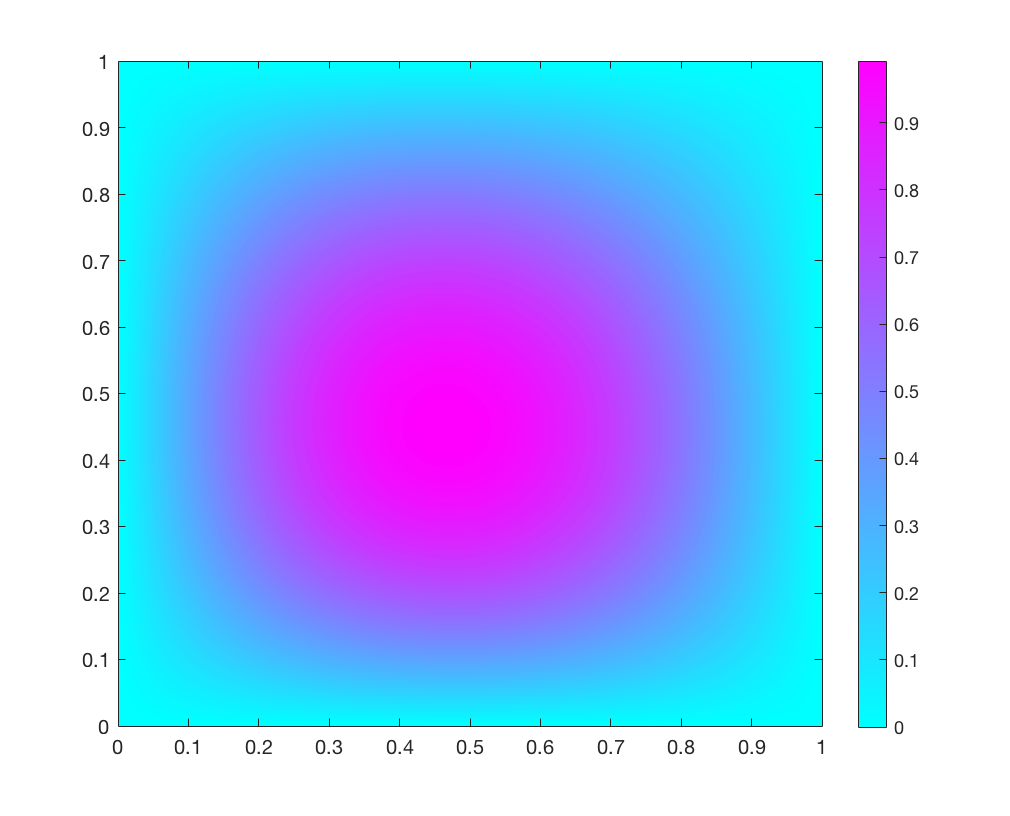}} \hspace{0.5cm}
\subfloat[eigenfunction for $\tau_{2,N}$]{\includegraphics[width=7cm]{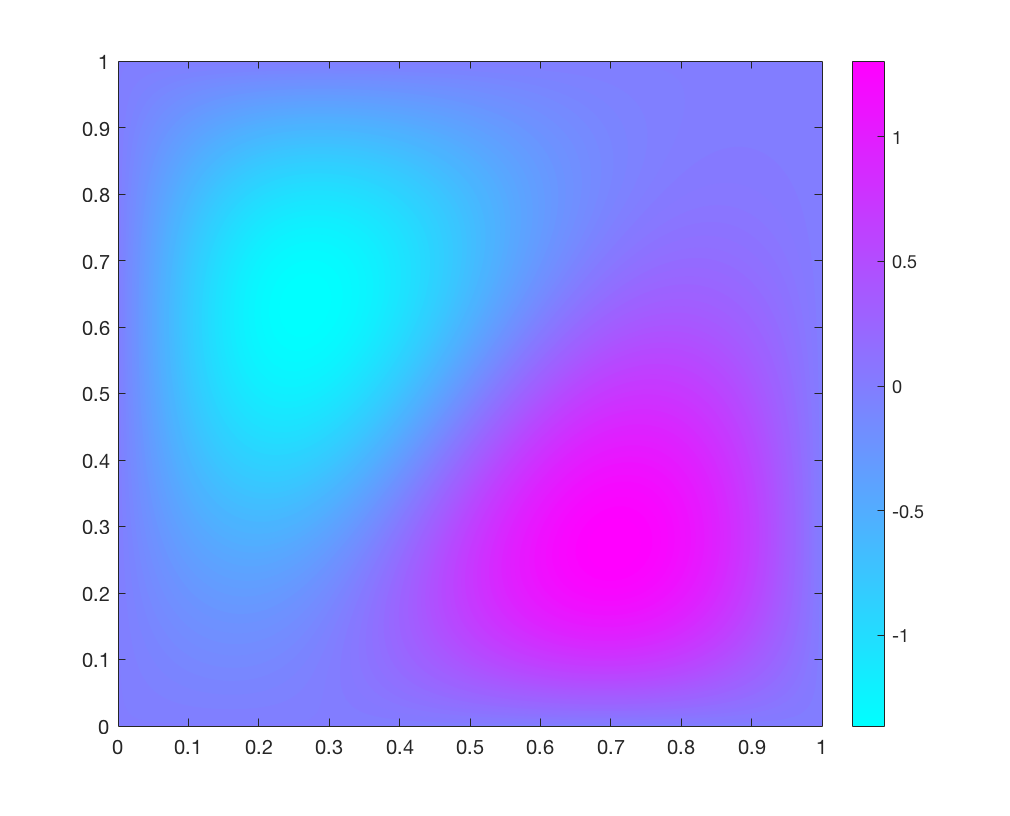}}\\
\subfloat[eigenfunction for $\tau_{3,N}$]{\includegraphics[width=7cm]{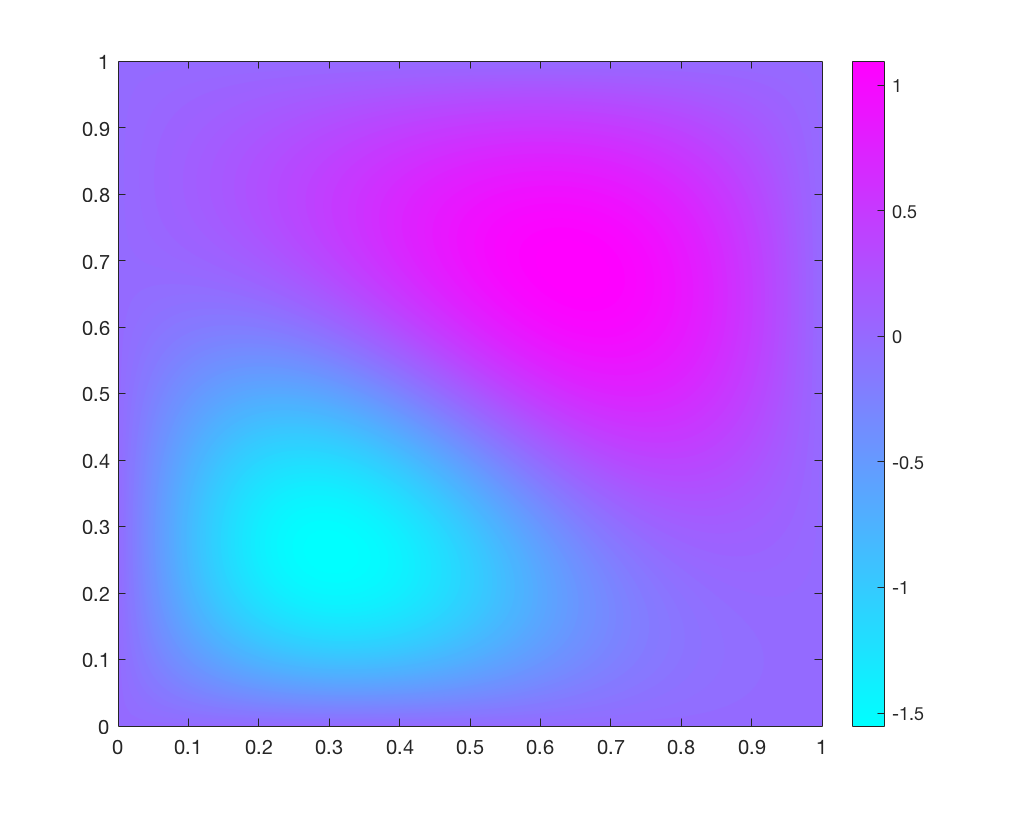}}\hspace{0.1in}
\subfloat[log-log plot of rel. error for $\tau_{j,N}$]{\includegraphics[width=7cm]{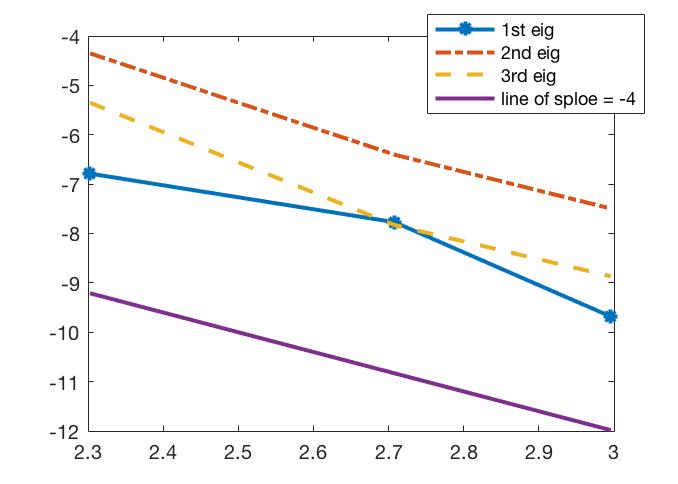}}
\caption{Plots of the first three eigenfunctions for the uint square for $N=25$ with coefficients $\alpha=x_1 x_2+1/4$ and $\beta=(x_1^2+1)(x_2^2+10)$. Also reported is the log-log plot of the relative error for $\tau_{j,N}$ compared to a line with slope $-4$.  } \label{eigfunc}
\end{figure}

\section{Summary and Conclusions}\label{last}

In conclusion, we have analyzed and implemented the Dirichlet spectral-Galerkin method to compute the simply supported vibrating plate eigenvalues. The analysis uses the fact that the Dirichlet eigenfunctions for the Laplacian have sufficient approximation properties which are estimated using Wely's law. From our numerical experiments we see that the convergence rate is approximately $\mathcal{O}(N^{-4})$. More work is needed to prove the fourth-order convergence of the eigenvalues. The main assumption on the domain $D$ is that it's boundary $\partial D$ be either polygonal with no reentrant corners or class $\mathscr{C}^2$. This is mainly to ensure that the Dirichlet eigenfunctions form a basis in the solution space $X(D)$. From our examples, we see that approximation seems to compute the eigenvalues fairly accurately for a modest size matrix. As stated in the previous sections, for domains where the  Dirichlet eigenfunctions are known analytically this method is simple to implement and gives good approximations for a modest amount of basis functions. Lastly, one can also attempt to compute other simply supported eigenvalues with this method such as the bucking plate eigenvalues. See \cite{fem-4th2} for the mixed finite element method applied to this eigenvalue problem.   \\

\noindent{\bf Acknowledgments:} The author would like to thank Andreas Kleefeld for helpful feedback and suggestions which lead to a much improved manuscript. 


\end{document}